\documentclass[sts]{imsart}

\RequirePackage{amsthm,amsmath,amsfonts,amssymb}
\RequirePackage[authoryear]{natbib}
\RequirePackage[colorlinks,citecolor=blue,urlcolor=blue]{hyperref}

\usepackage{amssymb}
\usepackage{dsfont}
\usepackage{bigints}

\startlocaldefs
\numberwithin{equation}{section}
\theoremstyle{plain}
\newtheorem{theorem}{Theorem}
\newtheorem{lemma}[theorem]{Lemma}
\newtheorem{corollary}[theorem]{Corollary}

\theoremstyle{remark}
\newtheorem{assumption}[theorem]{Assumption}
\newtheorem{definition}[theorem]{Definition}
\newtheorem{remark}[theorem]{Remark}
\newtheorem*{fact}{Fact}

\renewcommand{\leq}{\leqslant}
\renewcommand{\geq}{\geqslant}
\newcommand{\bone}{\boldsymbol{1}}
\newcommand{\defeq}{\stackrel{\mbox{\scriptsize \rm def}}{=}}
\newcommand{\R}{\mathbb{R}}
\renewcommand{\L}{\mathbb{L}}
\newcommand{\E}{\mathbb{E}}
\renewcommand{\H}{\mathbb{H}}
\renewcommand{\P}{\mathbb{P}}
\newcommand{\Q}{\mathbb{Q}}
\newcommand{\M}{\mathbb{M}}
\newcommand{\cP}{\mathcal{P}}
\newcommand{\cF}{\mathcal{F}}
\newcommand{\cT}{\mathcal{T}}
\newcommand{\cX}{\mathcal{X}}
\newcommand{\cS}{\mathcal{S}}

\newcommand{\cB}{\mathcal{B}}

\newcommand{\cN}{\mathcal{N}}
\newcommand{\cM}{\mathcal{M}}
\renewcommand{\d}{\mathrm{d}}
\newcommand{\dxi}{\dot{\xi}}
\newcommand{\Supp}{\mathrm{Supp}}
\newcommand{\transp}{\mbox{\rm \tiny T}\,}
\newcommand{\norm}[1]{\ensuremath{\Arrowvert #1 \Arrowvert}}
\newcommand{\bnorm}[1]{\ensuremath{\bigl\Arrowvert #1 \bigr\Arrowvert}}

\newcommand{\IFP}{\mathcal{I}_{\cP}}
\newcommand{\IFPn}{\mathcal{I}_{\cP^{\otimes n}}}
\newcommand{\IFPalt}{\widetilde{\mathcal{I}}_{\cP}}
\newcommand{\IFQ}{\mathcal{I}_{\Q}}
\newcommand{\IFQb}{\mathcal{I}_{\Q_{\theta_0,1}}}
\newcommand{\Leb}{\mathfrak{m}}
\newcommand{\Tr}{\mathop{\mathrm{Tr}}}
\newcommand{\Var}{\mathrm{Var}}

\endlocaldefs

\begin{document}

\begin{frontmatter}
\title{The van Trees inequality \\ in the spirit of H{\'a}jek and Le Cam}
\runtitle{The van Trees inequality in the spirit of H{\'a}jek and Le Cam}

\begin{aug}
\author[A]{\fnms{Elisabeth}~\snm{Gassiat}\ead[label=e1]{elisabeth.gassiat@universite-paris-saclay.fr}}
\and
\author[B]{\fnms{Gilles}~\snm{Stoltz}\ead[label=e2]{gilles.stoltz@universite-paris-saclay.fr}}
\address{Elisabeth Gassiat is Professor
and Gilles Stoltz is CNRS Senior research fellow;
Universit{\'e} Paris-Saclay, CNRS, Laboratoire de
math{\'e}matiques d'Orsay, 91405, Orsay, France\printead[presep={\ }]{e1,e2}.}
\end{aug}

\begin{abstract}
In honor of the 100th birth anniversary of Lucien Le Cam
(November 18, 1924 -- April 24, 2000),
we work out a version of the van Trees inequality in a H{\'a}jek--Le Cam spirit,
i.e., under minimal assumptions that, in particular, involve no direct pointwise regularity
assumptions on densities but rather almost-everywhere differentiability in quadratic mean
of the model. Surprisingly, it suffices that the latter differentiability
holds along canonical directions---not along all directions.
Also, we identify a (slightly stronger) version of the van Trees inequality
as a very instance of a Cram{\'e}r--Rao bound, i.e., the van Trees inequality
is not just a Bayesian analog of the Cram{\'e}r--Rao bound.
We provide, as an illustration, an elementary proof of the
local asymptotic minimax theorem for quadratic loss functions,
again assuming differentiability in quadratic mean only along canonical directions.
\end{abstract}


\end{frontmatter}

\section{Introduction}

Every statistician knows about the Cram{\'e}r--Rao inequality
but fewer knew about the van Trees inequality
(\citealp[page~72]{vT68}) before \citet{GiLe95} drew attention
to some of its statistical uses. In their landmark article,
they present the van Trees inequality as offering a Bayesian
Cram{\'e}r--Rao bound, to be applied in cases involving convergence of experiments
to bypass the beautiful but sophisticated H{\'a}jek--Le Cam theory of convergence of experiments.
\citet{GiLe95} derived the van Trees inequality under precise analytic conditions,
involving, in particular, smoothness assumptions on the densities;
so did also later contributions, including the ones by
\citet{Len05}, \citet{Jup10}, and \citet{Let22}.
However, as summarized by \citet{Pol01}, who in turn refers to
\citet[page~12]{Bic93} and \citet[Chapter~12]{LehmannRomano2005},
Le Cam and H{\'a}jek advocated resorting rather to conditions
that are intrinsic; of particular interest, is the concept of
differentiability in quadratic mean of a statistical model.

We provide a version of the van Trees inequality in the spirit
promoted by Le Cam and H{\'a}jek, and aim for
the weakest possible assumptions. In the one-dimensional
case (Section~\ref{sec:vT1-statement}), on top of the assumptions merely ensuring the existence of
the quantities involved in the inequality
(which includes the almost-everywhere differentiability of the model), we only require
that the prior vanishes at finite boundary points
of the parameter space~$\Theta$ (which is an arbitrary, not necessarily bounded, open subset of $\R$),
together with some technical condition on the model that is weaker than its differentiability everywhere.
We discuss these extremely mild assumptions (Section~\ref{sec:comp-classic}) by comparing them to the classic
regularity assumptions proposed by~\citet{GiLe95}.
Our proof (Section~\ref{sec:proof-vt1}) also exploits the same separation
of $x$ and $\theta$ variables as in~\citet{GiLe95}, but we perform
integrations in the reverse order, first over $x$ then over $\theta$,
thus effectively avoiding pointwise regularity assumptions on densities.
It turns out (Section~\ref{sec:CR-joint})
that the van Trees inequality is not only a Bayesian analog of the Cram{\'e}r--Rao bound,
as pointed out by \citet[page~72]{vT68} and \citet{GiLe95},
but that it is exactly, at least in a slightly stronger form,
an instance of a Cram{\'e}r--Rao bound for a suitably chosen location model.

The rest of this contribution focuses on a multivariate version
of the van Trees inequality. We provide (Section~\ref{sec:multivar})
weak conditions that only involve differentiability in quadratic mean
of the model along canonical directions, not all directions.
We illustrate (Section~\ref{sec:LAM-direct}) the application
of this multivariate version to establish
a local asymptotic minimax theorem for quadratic loss functions.

\section{One-dimensional version}

We consider a statistical model $\cP = (\P_\theta)_{\theta \in \Theta}$, defined on a measurable
space $(\cX,\cF)$ and indexed by an open subset $\Theta$ of $\R$ (not necessarily an interval).
We assume that $\cP$ is dominated by a $\sigma$--finite measure~$\mu$, with densities $f_\theta=\d\P_\theta/\d\mu$
such that $(\theta,x) \mapsto f_\theta(x)$ is measurable.
Let $\xi_\theta = \sqrt{f_\theta} \in \L^2(\mu)$ be the square roots of these densities.

In the sequel, $\norm{\cdot}_\mu$ refers to the Euclidean norm in $\L^2(\mu)$, i.e.,
for a function $g : \cX \to \R$ in $\L^2(\mu)$,
\[
\norm{g}_\mu = \sqrt{\int_{\cX} g^2 \,\d\mu}\,.
\]

\begin{definition}[{Differentiability in $\L_2$}]
\label{def:DMQ}
The $\mu$--{do\-mi\-na\-ted} statistical model $\cP$ is differentiable in $\L_2(\mu)$ at
$\theta_0 \in \Theta$ if there exists a function $\dxi_{\theta_0} \in \L_2(\mu)$, called the
$\L_2(\mu)$-derivative of the model at $\theta_0$, such that
\[
\bnorm{\xi_{\theta} - \xi_{\theta_0} - (\theta - \theta_0) \dxi_{\theta_0}}_\mu =
o \bigl( \norm{\theta - \theta_0} \bigr) \ \mbox{as} \ \theta \to \theta_0\,.
\]
The Fisher information $\IFP(\theta_0)$ of the model at $\theta_0$ is then defined
as
\[
\IFP(\theta_0) = 4 \int_\cX \bigl( \dxi_{\theta_0} \bigr)^2 \,\d\mu\,.
\]
\end{definition}

\begin{definition}[Well-behaved prior]
\label{def:well-behaved-1dim}
We call a probability measure $\Q$ that concentrates on the open set~$\Theta \subseteq \R$ a well-behaved prior if
$\Q$ has a density $q$ with respect to the Lebesgue measure on $\Theta$
that is absolutely continuous on $\Theta$, with almost-sure derivative~$q'$
satisfying
\[
\IFQ \defeq \int_\Theta \bigl( q'(\theta) \bigr)^2 \, \frac{\bone_{\{q(\theta)>0\}}}{q(\theta)}\,\d\theta <\infty\,.
\]
We denote by $\Supp(q) = \{ q > 0 \}$ the open support of $q$.
\end{definition}

A standard result (see, e.g., \citealp[Corollary~12.2.1]{LehmannRomano2005})
states that a location model based on a well-behaved prior $\Q$ is differentiable in $\L^2(\lambda)$,
where $\lambda$ denotes the Lebesgue measure,
with derivative at $0$ equal to $q' \bone_{\{q>0\}} / \bigl( 2\sqrt{q} \bigr)$,
and hence, with Fisher information $\IFQ$.

\subsection{Statement}
\label{sec:vT1-statement}

The van Trees inequality
lower bounds the Bayesian squared error of any, possibly biased,
statistic $S : \cX \to \R$ for the estimation of a functional $\psi(\theta)$,
where we assume that $\psi$ is an absolutely continuous function,
with almost-everywhere derivative denoted by $\psi'$.
More precisely, denoting by $\E_\theta$ the expectation under $\P_\theta$,
the one-dimensional version of the van Trees inequality reads
\begin{equation}
\label{eq:vT1}
\tag{vT1}
\bigintsss_{\Theta} \E_\theta \Bigl[ \bigl(S-\psi(\theta)\bigr)^2 \Bigr] \,\d\Q(\theta)
\geq \frac{\displaystyle{\left( \int_{\Theta} \psi'(\theta) \, \d\Q(\theta) \right)^{\!\! 2}}}{\displaystyle{\IFQ + \int_{\Theta} \IFP(\theta) \,\d\Q(\theta)}}\,.
\end{equation}
Our version of the van Trees inequality requires two series of assumptions.
The first series, stated in Assumption~\ref{ass:def}
merely ensures that all quantities involved are defined and that
the inequality has a meaning.
The second series of assumptions are ``real'' assumptions and may
be found in Theorem~\ref{th:vT1}.

\begin{assumption}[ensuring definitions and meaning]
\label{ass:def}
The set $\Theta$ is any open subset of $\R$.
The probability measure $\Q$ is a well-behaved prior on $\Theta$.
The statistical model  $\cP = (\P_\theta)_{\theta \in \Theta}$
is dominated by a $\sigma$--finite measure~$\mu$, with densities $f_\theta=\d\P_\theta/\d\mu$
such that $(\theta,x) \mapsto f_\theta(x)$ is measurable.
The model $\cP$ is differentiable in $\L^2(\mu)$ almost everywhere on $\Theta \cap \Supp(q)$.
The function $\psi : \Theta \to \R$ is absolutely continuous.
Both $\psi^2$ and $\psi'$ are $\Q$--integrable and
\[
\int_{\Theta} \E_\theta\bigl[S^2\bigr] \,\d\Q(\theta) < +\infty\,, \ \
\int_{\Theta} \IFP(\theta) \,\d\Q(\theta) < +\infty\,.
\]
\end{assumption}

\begin{theorem}
\label{th:vT1}
The one-dimensional van Trees inequality~\eqref{eq:vT1}
holds with $\IFQ > 0$ under Assumption~\ref{ass:def} and the following
additional assumptions:
\begin{itemize}
\item for all $A \in \cF$,
the functions $\theta \in \Theta \cap \Supp(q) \mapsto \P_\theta(A)$ are absolutely continuous;
\item $q(\theta) \to 0$ as $\theta$ approaches any finite boundary point of $\Theta$.
\end{itemize}
The first assumption holds in particular if the model $\cP$
is differentiable in $\L^2(\mu)$ at \emph{all} points of $\Theta \cap \Supp(q)$,
not just almost everywhere.
\end{theorem}

\subsection{Comparison to classic regularity assumptions}
\label{sec:comp-classic}

We compare Theorem~\ref{th:vT1}
to the version under classic regularity assumptions
by \citet{GiLe95} based on \citet{vT68}.
With no loss of generality (on the contrary) and no change in their proof, we only replace their closed
interval $\Theta$ by any open set $\Theta$, possibly intersected
with $\Supp(q)$. The key additional assumption required is stated next.

\begin{assumption}[main regularity assumption]
\label{ass:ACf}
In the $\mu$--dominated model $\cP$,
the densities $f_\theta=\d\P_\theta/\d\mu$ are such that
for $\mu$--almost all $x$,
the function $\theta \in \Theta \cap \Supp(q) \mapsto f_\theta(x)$ is absolutely continuous,
with almost-everywhere derivative denoted by $f'_\theta(x)$.
\end{assumption}

In that setting with classic regularity assumptions, the Fisher information is
defined, where $f'_\theta$ exists, i.e., almost-everywhere, by
\[
\IFPalt(\theta) =
\bigintsss_\cX \biggl( \frac{f'_\theta}{f_\theta} \biggr)^{\! 2} \,\d\P_\theta
= \bigintsss_\cX \frac{(f'_\theta)^2}{f_\theta} \bone_{\{f_\theta > 0\}} \,\d\mu\,.
\]
A finite denominator in the right-hand side of the van Trees inequality
entails (see the argument in the last lines of
Section~\ref{sec:loc-int}) that $\IFPalt$ is locally integrable around each $\theta \in \Theta \cap \Supp(q)$,
and thus, that almost all points of $\Theta \cap \Supp(q)$ are Lebesgue points for $\IFPalt$.
Based on this and on Assumption~\ref{ass:ACf}, we apply a slight extension of \citet[Proposition~1]{Bic93}
or \citet[Theorem~12.2.1]{LehmannRomano2005}, whose proofs show that
continuity of $\IFPalt$ is actually not required and that a Lebesgue-point assumption is sufficient;
we obtain that the model $\cP$ is differentiable in $\L^2(\mu)$ almost everywhere on $\Theta \cap \Supp(q)$,
with $\L^2(\mu)$--derivatives given by $\dxi_\theta = f'_\theta \bone_{\{f_\theta > 0\}} / \sqrt{f_\theta}$.
We also have $\smash{\IFPalt} = \IFP$ almost everywhere on $\Theta \cap \Supp(q)$.

Now, \citet{GiLe95} prove the van Trees inequality
under the boundary conditions on $q$ and $q\psi$ stated in Theorem~\ref{th:vT1},
under Assumption~\ref{ass:ACf} and all of Assumption~\ref{ass:def}
except the almost everywhere $\L^2(\mu)$--differentiability of $\cP$.
The other condition in Theorem~\ref{th:vT1},
namely, that for all $A \in \cF$,
the function $\theta \in \Theta \cap \Supp(q) \mapsto \P_\theta(A)$ is absolutely continuous,
is a direct consequence of Assumption~\ref{ass:ACf}, by the Fubini--Tonelli theorem
and the characterization of absolute continuity in terms of equality to the integral of the derivative.
We therefore proved the following fact.

\begin{fact}
The regularity assumptions considered by \citet{GiLe95} to prove the
one-dimensional van Trees inequality~\eqref{eq:vT1}
are more stringent than the
H{\'a}jek--Le Cam-type assumptions considered in Theorem~\ref{th:vT1}.
\end{fact}

\subsection{Proof of Theorem~\ref{th:vT1}}
\label{sec:proof-vt1}

The key lemma for our approach and its proof are
extracted from the lecture notes by \citet{Pol01},
who adapted a result by~\citet[Lemma~7.2, page~67]{IbHa81}.
The lemma stated in \citet{Pol01} is actually stronger as it only requires
local boundedness of $T$ in $\L^2(\P_\theta)$ around~$\theta_0$.

\begin{lemma}[\citealp{Pol01}]
\label{lm:info-eq}
Let the $\mu$--dominated model $\cP = (\P_\theta)_{\theta \in \Theta}$ be differentiable in $\L_2(\mu)$ at $\theta_0$.
Consider a uniformly bounded statistic $T : \cX \to \R$, i.e., there exists $M > 0$ with $|T| \leq M$ $\mu$--a.s.
Then, $\gamma_T : \theta \in \Theta \mapsto \E_{\theta}[T]$
is differentiable at $\theta_0$, with derivative
\[
\gamma'_T(\theta_0) = 2 \int_\cX \dxi_{\theta_0} \xi_{\theta_0} T \, \d\mu\,.
\]
\end{lemma}

\begin{proof}
Let $r_\theta = \xi_\theta - \xi_{\theta_0} - (\theta - \theta_0) \dxi_{\theta_0}$, so that
\vspace{-.35cm}

\begin{multline*}
\overbrace{\bigl( \xi_{\theta_0} + (\theta - \theta_0) \dxi_{\theta_0} + r_\theta \bigr)^2}^{= \, \xi_\theta^2 \, = \, f_\theta} - \xi_{\theta_0}^2
- 2 (\theta-\theta_0) \dxi_{\theta_0} \xi_{\theta_0} \\
= (\theta - \theta_0)^2 \dxi_{\theta_0}^2 + r_\theta^2 + 2 r_\theta \xi_{\theta_0} + 2 (\theta - \theta_0) \dxi_{\theta_0} r_\theta
\end{multline*}
The $\L^1(\mu)$--norms of the first two terms in the right-hand side is of
order $(\theta - \theta_0)^2$. The $\L^1(\mu)$--norms of the last two terms above are (by the Cauchy--Schwarz
inquality) of order $\norm{r_\theta}_\mu$, thus are $o\bigl( |\theta - \theta_0| \bigr)$.
Multiplying both sides of the display above by the bounded $T$ and integrating over $\mu$, we obtain
\[
\gamma_T(\theta) - \gamma_T(\theta_0) - 2(\theta-\theta_0) \int_\cX \dxi_{\theta_0} \xi_{\theta_0} T \, \d\mu
= o\bigl( |\theta - \theta_0| \bigr)\,. \vspace{-.2cm}
\]

\end{proof}

\subsubsection{Overview of the proof.} We introduce
\[
\Delta : (x,\theta)
\longmapsto q'(\theta) \frac{\bone_{\{q(\theta) > 0\}}}{2\sqrt{q(\theta)}} \, \xi_\theta(x) + \sqrt{q(\theta)} \, \dxi_\theta(x)\,,
\]
which is well-defined for almost all $\theta \in \Theta \cap \Supp(q)$, and vanishes for $\theta \not\in \Supp(q)$.
Let $\Leb$ denote the Lebesgue measure.
We will show that
\begin{align}
\nonumber
2 \int_{\Theta \times \cX} \Delta(x,\theta) \, \sqrt{q(\theta)} \, \xi_\theta(x) \, \bigl( S(x) & - \psi(\theta) \bigr) \,\d\theta \, \d\mu(x)
\\
\label{eq:vT-infoeq}
& = \int_\Theta \psi'(\theta) \, \d \Q(\theta)\,.
\end{align}
We prove the equality~\eqref{eq:vT-infoeq} above in a direct way, and the van Trees inequality
then follows by an application of the Cauchy--Schwarz inequality.
Section~\ref{sec:CR-joint} explains that~\eqref{eq:vT-infoeq} can actually be interpreted, under
stronger assumptions, as a consequence of
Lemma~\ref{lm:info-eq} with $T(x,\theta) = S(x)-\psi(\theta)$ and a well-chosen location model.
Actually, a close look at the proof by \citet[page~61]{GiLe95} shows that
they also exactly prove~\eqref{eq:vT-infoeq}, though under additional
regularity assumptions, like the $\theta \mapsto f_\theta(x)$ being absolutely continuous,
and by first integrating in the left-hand side over $\theta$ then over $x$.
We take the reverse order and first integrate over $x$, thanks to
applications of Lemma~\ref{lm:info-eq}, and then over~$\theta$.

\subsubsection{Preparations.}
\label{sec:prep-vT1}
It suffices to prove~\eqref{eq:vT1} for statistics $S$ given by finite linear combinations
of indicator functions, the case of general statistics following by taking limits
given the bounded second moment stated in Assumption~\ref{ass:def}.
Similarly, the sequence of absolutely
continuous functions $\psi_n = \max\bigl\{-n,\min\{\psi,n\}\bigr\}$
satisfies $\psi_n \to \psi$ and $\psi'_n \to \psi'$ almost-surely;
by dominated convergence, it also suffices to prove~\eqref{eq:vT1}
for bounded $\psi$ with bounded derivatives.

The first assumption of Theorem~\ref{th:vT1} ensures that
the function $\gamma_S$ is absolutely continuous on $\Theta \cap \Supp$.
In addition, Lemma~\ref{lm:info-eq}, based on the fact that $\cP$ is differentiable
at almost all $\theta \in \Theta \cap \Supp(q)$ and that $S$ is in particular
uniformly bounded, provides a closed-form
expression for the almost-everywhere derivative $\gamma'_S$.

Finally, all integrands below belong to $\L^1(\Leb \otimes \mu)$,
as follows from applications of
the Cauchy--Schwarz inequality. Hence, integrals of sums equal sums of integrals and Fubini's theorem may be applied
to exchange orders of integration.
We use the short-hand notation $\mu[f]$ for the expectation of a function $f : \cX \to \R$ under~$\mu$.

\subsubsection{Proof of~\eqref{eq:vT-infoeq}.}
\label{eq:IPP-1dim}
Let $\Theta_q = \Theta \cap \Supp(q)$. The integrals
in~\eqref{eq:vT-infoeq} may be equivalently taken over $\Theta$ or $\Theta_q$.
The left-hand side of~\eqref{eq:vT-infoeq} consists of four terms,
namely,
\begin{align*}
\int_{\Theta_q} q'(\theta) \, \mu\bigl[f_\theta S\bigr] \,\d\theta
& = \int_{\Theta_q} q'(\theta) \, \gamma_S(\theta) \,\d\theta\,\\
- \int_{\Theta_q} \psi(\theta) \, q'(\theta) \, \mu\bigl[f_\theta\bigr] \,\d\theta
& = - \int_{\Theta_q} \psi(\theta) \, q'(\theta)\, \d\theta\,, \\
2 \int_{\Theta_q} q(\theta) \, \mu \bigl[\dxi_\theta \xi_\theta S\bigr] \,\d\theta
& = \int_{\Theta_q} q(\theta) \, \gamma'_S(\theta) \,\d\theta\,, \\
- 2 \int_{\Theta_q} \psi(\theta) \, q(\theta) \, \mu \bigl[\dxi_\theta \xi_\theta \bigr] \,\d\theta
& = 0\,.
\end{align*}
The fourth equality follows from Lemma~\ref{lm:info-eq} with $T \equiv 1$,
which entails that $\mu \bigl[\dxi_\theta \xi_\theta \bigr] = 0$
for almost all $\theta \in \Theta_q$.
Now, the functions $\gamma_S$, $q$ and $\psi$ are absolutely continuous on $\Theta_q$,
so that an integration by parts (\citealp[page 375, {\S}12.11]{Tit39}) ensures that
on any compact sub-interval $[c,d] \subset \Theta_q$,
\begin{align*}
\int_{[c,d]} \bigl( q'(\theta) \, \gamma_S(\theta) + q(\theta) \, \gamma'_S(\theta) \bigr) \,\d\theta
= \bigl[ q(\theta)\,\gamma_S(\theta) \bigr]_c^d\,, \\
\int_{[c,d]} \psi(\theta) \, q'(\theta)\, \d\theta = \bigl[ \psi(\theta) \,q(\theta) \bigr]_c^d - \int_{[c,d]} \psi'(\theta) \, q(\theta)\,\d\theta\,.
\end{align*}
We write $\Theta_q = \Theta \cap \Supp(q)$ as a countable union of disjoint intervals $(a_\tau, b_\tau)$, indexed by~$\tau \in \cT$.
Each finite boundary point of $\Theta_q$ is either a finite boundary point of $\Theta$,
or lies in the interior of $\Theta$ and is a finite boundary point of $\Supp(q)$;
in the latter case, by continuity,
$\psi$ is bounded and $q$ vanishes thereat.
Therefore, by boundedness of $\gamma_S$ and $\psi$ and by the $\Theta$--boundary assumptions on $q$,
the quantities $\gamma_S(\theta) \, q(\theta)$ and $\psi(\theta) \, q(\theta)$
vanish as $\theta$ approaches any finite boundary point $a_\tau$ or $b_\tau$ of~$\Theta_q$.
When $\pm \infty$ is a boundary point of $\Theta_q$, given that $q$ is integrable over $\Theta$,
the liminf of $q(\theta)$ is null as $\theta$ tends to $\pm \infty$.
Therefore, by boundedness of $\gamma_S$ and $\psi$ again,
for each $\tau$, by letting $c \to a_\tau$ and $d \to b_\tau$ in a suitable manner
and by dominated convergence, we have
\begin{align*}
\int_{(a_\tau,b_\tau)} \bigl( q'(\theta) \, \gamma_S(\theta) + q(\theta) \, \gamma'_S(\theta) \bigr) \,\d\theta
= 0\,, \\
\int_{(a_\tau,b_\tau)} \psi(\theta) \, q'(\theta)\, \d\theta = - \int_{(a_\tau,b_\tau)} \psi'(\theta) \, q(\theta)\,\d\theta\,.
\end{align*}
By dominated convergence, summing these inequalities over $\tau \in \cT$ yields \eqref{eq:vT-infoeq}.

\subsubsection{Conclusion by a Cauchy--Schwarz inequality.}
The van Trees inequality~\eqref{eq:vT1} follows by
applying the Cauchy--Schwarz inequality to~\eqref{eq:vT-infoeq}
together with the fact that
\[
4 \int_{\Theta \times \cX} \Delta(x,\theta)^2 \, \d\theta \d\mu(x) = \IFQ + \int_{\Theta} \IFP(\theta) \,\d\Q(\theta) \,.
\]
The equality above follows from the definitions of Fisher information (for
the integrals of square terms) and the fact that the following integral (corresponding to the cross term) is null,
since $\mu \bigl[\dxi_\theta \xi_\theta \bigr] = 0$ for almost all $\theta$, as already noted above:
\[
\int_{\Theta \times \cX} q'(\theta) \bone_{\{q(\theta) > 0\}} \, \dxi_\theta(x) \xi_\theta(x) \,\d\theta \, \d\mu(x) = 0\,.
\]
That $\IFQ > 0$ follows from the impossibility of $q$ to be a uniform distribution,
because of the vanishing-at-the-border constraints.
This concludes the proof of the first part of Theorem~\ref{th:vT1}
and we now move to its last statement.

\subsubsection{Special case.}
\label{sec:loc-int}
We finally show that when the model $\cP$ is $\L^2(\mu)$--differentiable
at all points of $\Theta_q$, not just almost everywhere,
the first assumption of Theorem~\ref{th:vT1} holds, namely, that
for all events $A\in {\cal F}$, the functions $\gamma_A : \theta\mapsto \P_{\theta}(A)$ are absolutely continuous
on $\Theta_q$. Indeed, by \citet[page 368, {\S}11.83]{Tit39},
it suffices to note that $\gamma_A$ is differentiable everywhere on $\Theta_q$
(by Lemma~\ref{lm:info-eq} together with the assumption that the
model $\cP$ is $\L^2(\mu)$--differentiable everywhere),
with a derivative $\gamma'_A$ that is finite everywhere and
locally integrable on $\Theta_q$: by the Cauchy--Schwarz inequality,
\[
\bigl| \gamma'_A(\theta) \bigr| = \left| 2 \int_\cX \dxi_{\theta} \xi_{\theta} \bone_A \, \d\mu \right|
\leq \sqrt{\IFP(\theta)} < +\infty\,.
\]
The claimed local integrability follows from the bound
above, the local integrability of $\IFP\,q$ (by
Assumption~\ref{ass:def}), and the fact that by absolute continuity,
$q \geq \delta$ for some $\delta > 0$ on
any open interval $(a,b) \subseteq \Theta_q$.

\subsection{The van Trees inequality as a Cram{\'e}r--Rao bound}
\label{sec:CR-joint}

The Cram{\'e}r--Rao bound (for possibly biased statistics~$T$)
is obtained as a corollary of Lemma~\ref{lm:info-eq}. By applying the Cauchy--Schwarz
inequality to the equality
\[
\gamma'_T(\theta_0) = 2 \int_\cX \dxi_{\theta_0} \xi_{\theta_0} T \, \d\mu\,,
\]
we get indeed, when $\IFP(\theta_0) > 0$,
\[
\E_{\theta_0}\bigl[T^2\bigr] \geq \frac{\bigl( \gamma'_T(\theta_0) \bigr)^2}{\IFP(\theta_0)}\,.
\]
Actually, replacing in the argument above $T$ by $T-c$, with $c = \E_{\theta_0}[T]$, yields the
desired Cram{\'e}r--Rao bound:
\[
\Var_{\theta_0}(T) = \E_{\theta_0}\bigl[(T-c)^2\bigr]
\geq \frac{\bigl( \gamma'_T(\theta_0) \bigr)^2}{\IFP(\theta_0)}\,.
\]

Now, the van Trees inequality was obtained in Section~\ref{sec:proof-vt1}
by an application of the Cauchy--Schwarz inequality to the equality~\eqref{eq:vT-infoeq},
which was claimed to be a consequence of Lemma~\ref{lm:info-eq}; this indicates
that the van Trees inequality is exactly an instance of a Cram{\'e}r--Rao bound (for the location model~$\cM$
described below), at least in the (slightly stronger) form of Corollary~\ref{cor:vT1} below.
The latter is an automatic improvement of Theorem~\ref{th:vT1},
as its proof merely consists of applying Theorem~\ref{th:vT1}
with $S-c$ and $\psi$ (or, alternatively, with $S$ and $\psi+c$)
for a well-chosen $c$.

\begin{corollary}
\label{cor:vT1}
Under the assumptions of Theorem~\ref{th:vT1}, we actually have the stronger lower bound
\begin{multline*}
\bigintsss_{\Theta} \E_\theta \Bigl[ \bigl(S-\psi(\theta)\bigr)^2 \Bigr] \,\d\Q(\theta) \geq \\
\left( \bigintsss_{\Theta} \E_\theta \bigl[ S-\psi(\theta) \bigr] \,\d\Q(\theta) \right)^{\!\! 2}
+
\frac{\displaystyle{\left( \int_{\Theta} \psi'(\theta) \, \d\Q(\theta)
\right)^{\!\! 2}}}{\displaystyle{\IFQ + \int_{\Theta} \IFP(\theta) \,\d\Q(\theta)}}\,.
\end{multline*}
\end{corollary}

Put differently, the van Trees inequality of Corollary~\ref{cor:vT1}
is not (only) to be understood as a Bayesian
Cram{\'e}r--Rao bound, as advocated by~\citet{GiLe95}, it is exactly a Cram{\'e}r--Rao bound.
Similarly, \citet[page~72]{vT68} underlines that he mimics the derivation of the Cram{\'e}r--Rao bound
to obtain his inequality, but does not see the latter as a very instance of the former.

We conclude this section by detailing our claim that
the equality~\eqref{eq:vT-infoeq} may be seen,
under suitable conditions (not required for our direct proof of Theorem~\ref{th:vT1}),
as a consequence of Lemma~\ref{lm:info-eq}.
We assume, in particular, that the support of $q$ is $\delta$--away from the border of $\Theta$, i.e.,
that for all $\theta \in \Theta$ with $q(\theta) > 0$ and all $x \in [-\delta,\delta]$, one has
$\theta + x \in \Theta$. This assumption ensures that
the location model $\cM = (\M_{\alpha})_{\alpha \in (-\delta,\delta)}$ is well defined, where
$\M_{\alpha}$ is the distribution over $\Theta \times \cX$ with density
$(x,\theta) \mapsto q(\theta+\alpha) f_{\theta+\alpha}(x)$ with respect to $\mu \otimes \Leb$.
Under suitable conditions (not detailed),
we may apply the same theorem as in Section~\ref{sec:comp-classic} (\citealp[Proposition~1]{Bic93}
or~\citealp[Theorem~12.2.1]{LehmannRomano2005})
establishing the $\L^2(\mu \otimes \Leb)$--differentiability of
$\M_{\alpha}$ at $\alpha_0=0$ and identifying its $\L^2(\mu \otimes \Leb)$--derivative at
$\alpha_0 = 0$, which we denote by $\Delta$,
with the pointwise derivative of
$(x,\theta) \mapsto \sqrt{q(\theta+\alpha) f_{\theta+\alpha}(x)}$ at $\alpha_0 = 0$:
\[
\Delta : (x,\theta)
\longmapsto q'(\theta) \frac{\bone_{\{q(\theta) > 0\}}}{2\sqrt{q(\theta)}} \, \xi_\theta(x) + \sqrt{q(\theta)} \, \dxi_\theta(x)\,.
\]
For a bounded statistic $S$ and an absolutely continuous and bounded target function $\psi$, whose
derivative $\psi'$ is also bounded,
we consider the statistic $J(x,\theta) = S(x) - \psi(\theta)$. Its expectation under
some $\M_{\alpha}$ equals
\begin{align*}
\lefteqn{\gamma_J(\alpha) = \E_{\M_\alpha}[J]} \\
& = \int_{\cX \times \Theta} \bigl( S(x) - \psi(\theta) \bigr) \,
q(\theta+\alpha) f_{\theta+\alpha}(x) \,\d\mu(x)\d\theta \\
& = \int_\Theta \E_{\theta}[S]\,q(\theta)\,\d\theta - \int_\Theta \psi(\theta-\alpha)\,q(\theta)\,\d\theta\,.
\end{align*}
Differentiating the above equality at $\alpha_0 = 0$, we obtain, as claimed,
the equality~\eqref{eq:vT-infoeq}, whose left-hand
side may be identified to $\gamma'_J(\alpha)$ thanks to
Lemma~\ref{lm:info-eq}, and whose right-hand side is obtained by differentiating under
the integral sign.

\section{Multivariate version}
\label{sec:multivar}

There exist several ways to extend the van Trees inequality
for multivariate estimation; see \citet{GiLe95}, who in turn refer
to \citet{vT68} and \citet{BMZ87}.
We focus here on the elegant matrix-wise version by~\citet{Let22}.

Let the statistical model $\cP = (\P_\theta)_{\theta \in \Theta}$ be indexed by an open
set $\Theta \subseteq \R^p$, where $p \geq 2$.
The estimation target will be some $\psi(\theta)$, where $\psi : \Theta \to \R^s$,
and we consider some statistic $S : \cX \to \R^s$ to that end.
We still assume that $\cP$ is dominated by a $\sigma$--finite measure~$\mu$, with densities $f_\theta=\d\P_\theta/\d\mu$
such that $(\theta,x) \mapsto f_\theta(x)$ is measurable.
In the sequel, $\norm{\cdot}$ refers to the Euclidean norm in some $\R^d$ space
(with $d = p$ or $d = s$),
and $\norm{\cdot}_\mu$ denotes the Euclidean norm in $\L^2(\mu)$, i.e.,
for a function $g : \cX \to \R^d$ in $\L^2(\mu)$,
\[
\norm{g}_\mu = \sqrt{\int_{\cX} \norm{g}^2 \,\d\mu}\,.
\]

\subsection{Comparison to classic regularity assumptions}

Both~\citet{GiLe95} and~\citet{Let22}
assume some smoothness on the functions $\theta \mapsto f_\theta(x)$,
for $\mu$--almost all $x$, and also possibly on the border
of $\Theta$. These assumptions are useful to
extend the integrations by parts performed in Section~\ref{eq:IPP-1dim}
to the multivariate case, via Stokes' theorem.
More precisely, \citet{Let22} assumes (this is what he calls a ``regular
Fisher model'') that the functions $\theta \mapsto f_\theta(x)$ are even
$C^1$--smooth but does not put any constraint on the boundary of
$\Theta$. \citet{GiLe95} assume, in particular, that $\Theta$
is compact with a piecewise-$C^1$--smooth boundary; as for the functions
$\theta \mapsto f_\theta(x)$, they assume that they are ``nice''
for $\mu$--almost all $x$ in the sense of Definition~\ref{def:nice}
(which is actually a property that Sobolev functions enjoy,
see \citealp[Section~4.9]{EvGa92}).
For $u = (u_1,\ldots,u_d) \in \R^d$ and $i \in \{1,\ldots,d\}$,
we let $u_{-i}$ denote the $(d-1)$--dimensional vector of all components of $u$ but the $i$--th one,
so that, by an abuse of notation, $u = (u_i,u_{-i})$.
We introduce the projection of a subset $D \subseteq \R^d$
ignoring the $i$--th coordinates:
\[
D_{-i} = \bigl\{ u_{-i} \in \R^{d-1} : \exists \, u_i \in \R \ \mbox{s.t.} \ (u_i,u_{-i}) \in D \bigr\}\,.
\]

\begin{definition}[nice functions]
\label{def:nice}
Let $D \subseteq \R^d$ be an open domain, where $d \geq 2$.
A function $\varphi : D \to \R$ is nice if for all $i \in \{1,\ldots,d\}$,
for almost all $u_{-i} \in D_{-i}$,
the functions $u_i \mapsto \varphi(u_i,u_{-i})$ are
absolutely continuous in the classic one-dimensional sense
on the open domain $D(u_{-i}) = \bigl\{ u_i \in \R : \ (u_i,u_{-i}) \in D \bigr\}$.

In particular, a function $\varphi : D \to \R$ that is nice
admits at almost all $u \in D$ partial derivatives along canonical
directions, which we denote by $\partial_1 \varphi, \ldots, \partial_d \varphi$.
By an abuse of notation, we denote by $\nabla \varphi =
(\partial_1 \varphi, \ldots, \partial_d \varphi)$ the vector of
partial derivatives.

A vector-valued function is nice if each of its component is nice.
\end{definition}

As in Section~\ref{eq:IPP-1dim}, we avoid
issuing regularity assumptions on the
functions $\theta \mapsto f_\theta(x)$
and replace them by $\L^2(\mu)$--differentiability
assumptions. Our version of the van Trees inequality
only requires such an $\L^2(\mu)$--differentiability
to hold along canonical directions, not all directions.
For the sake of a simpler exposition,
and as in the second part of Theorem~\ref{th:vT1},
we restrict our attention to a model that is
$\L^2(\mu)$--differentiable along canonical directions at all points.
We denote by $u \otimes v = u \, v^{\transp}$ the outer product of two vectors $u$ and $v$ (possibly of different lengths).

\begin{definition}[{Differentiability in $\L_2$ along canonical directions}]
\label{def:DMQ}
The $\mu$--{do\-mi\-na\-ted} statistical model $\cP$
indexed by an open subset $\Theta \subseteq \R^p$
is differentiable in $\L_2(\mu)$ at
$\theta_0 \in \Theta$ along canonical directions
if there exist scalar functions $\dxi_{\theta_0,1},
\ldots,\dxi_{\theta_0,p} \in \L_2(\mu)$, called the
$\L_2(\mu)$--partial derivatives of the model at $\theta_0$, such that,
for all $i \in \{1,\ldots,p\}$, as $\theta_i \to \theta_{0,i}$,
\[
\bnorm{\xi_{(\theta_i,\,\theta_{0,-i})} - \xi_{\theta_0} - (\theta_i - \theta_{0,i}) \dxi_{\theta_0,i}}_\mu =
o \bigl( |\theta_i - \theta_{0,i}| \bigr)\,.
\]
Let $\dxi_{\theta_0} = (\dxi_{\theta_0,1},\ldots,\dxi_{\theta_0,p})$.
The Fisher information $\IFP(\theta_0)$ of the model at $\theta_0$ is then defined as the
$p \times p$ matrix
\[
\IFP(\theta_0) = 4 \int_\cX \dxi_{\theta_0} \otimes \dxi_{\theta_0} \,\d\mu\,.
\]
\end{definition}

While we avoid at all costs direct
regularity assumptions on the
functions $\theta \mapsto f_\theta(x)$,
as we have no control on the model $\cP$,
we may be more lenient when it comes
to the prior $\Q$, which the statistician chooses.
\citet{GiLe95} impose, among others, the following
assumption on $\Q$, which generalizes Definition~\ref{def:well-behaved-1dim}.

\begin{definition}[Well-behaved prior, multivariate version]
We call a probability measure $\Q$ that concentrates on the open set~$\Theta \subseteq \R^p$ a well-behaved prior if
$\Q$ has a density $q$ with respect to the Lebesgue measure on $\Theta$
that is nice on $\Theta$, and whose vector of partial derivatives $\nabla q$
is such that $\Arrowvert \nabla q \Arrowvert^2_2 \, \bone_{\{q>0\}}/q$ is Lebesgue-integrable.
We define
\[
\IFQ \defeq \int_\Theta \nabla q (\theta) \otimes \nabla q (\theta) \, \frac{\bone_{\{q(\theta)>0\}}}{q(\theta)}\,\d\theta \,.
\]
\end{definition}

\subsection{Statement}

The multivariate version of the van Trees inequality
proposed by~\citet{Let22}, as well as a consequence thereof (in terms of
Schur complement) is stated in~\eqref{eq:vT-multi}. Therein,
where $M \succcurlyeq 0$ and $M \succ 0$ denote the fact that a symmetric matrix $M$ is
positive semi-definite and positive definite, respectively.
Also, $\nabla \psi(\theta)$ denotes the $p \times s$ matrix whose component
$(i,j)$ equals $\nabla \psi(\theta)_{i,j} = \partial_i \psi_j(\theta)$.

\begin{figure*}
\begin{equation}
\label{eq:vT-multi}
\tag{vTm}
\left[
\begin{array}{ccc}
\displaystyle{\int_{\Theta} \E_\theta \Bigl[ \bigl( S-\psi(\theta) \bigr) \otimes \bigl( S-\psi(\theta) \bigr) \Bigr] \, \d\Q(\theta)} & ~~ &
\displaystyle{\left( \int_{\Theta} \nabla{\psi}(\theta) \, \d\Q(\theta) \right)^{\! \transp}} \vspace{.125cm} \\
\displaystyle{\int_{\Theta} \nabla{\psi}(\theta) \, \d\Q(\theta)} & & \displaystyle{\IFQ + \int_{\Theta} \IFP(\theta) \, \d\Q(\theta)}
\end{array}
\right] \succcurlyeq 0\,,
\end{equation}
\flushleft thus, whenever $\IFQ \succ 0$,
\[
\int_{\Theta} \E_\theta \Bigl[ \bigl( S-\psi(\theta) \bigr) \otimes \bigl( S-\psi(\theta) \bigr) \Bigr] \, \d\Q(\theta)
- \left( \int_{\Theta} \nabla{\psi}(\theta) \, \d\Q(\theta) \right)^{\! \transp} \,
\left( \IFQ + \int_{\Theta} \IFP(\theta) \, \d\Q(\theta) \right)^{\! -1}
\left( \int_{\Theta} \nabla{\psi}(\theta) \, \d\Q(\theta) \right)
\succcurlyeq 0\,.
\]
\rule{\textwidth}{0.2pt}
\end{figure*}

The multivariate counterpart of Assumption~\ref{ass:def}
is stated next. It does not target generality
and aims to ease exposition: as a consequence,
it requires differentiability of the model at all points of $\Theta \cap \Supp(q)$,
not just almost everywhere, and also imposes that the density $q$ is continuous
(which does not follow from Definition~\ref{def:nice}).

\begin{assumption}[for the multivariate case]
\label{ass:def-multidim}
The set~$\Theta$ is any open subset of $\R^p$.
The probability measure~$\Q$ is a well-behaved prior on $\Theta$,
with a continuous density~$q$.
The statistical model $\cP = (\P_\theta)_{\theta \in \Theta}$
is dominated by a $\sigma$--finite measure~$\mu$, with densities $f_\theta=\d\P_\theta/\d\mu$
such that $(\theta,x) \mapsto f_\theta(x)$ is measurable.
The model $\cP$ is differentiable in $\L^2(\mu)$ along canonical
dimensions at \emph{all} points of on $\Theta \cap \Supp(q)$.
The function $\psi : \Theta \to \R^s$ is nice.
Both $\norm{\psi}^2$ and $\norm{\nabla \psi}$ are $\Q$--integrable and
both
\[
\int_{\Theta} \E_\theta\bigl[\norm{S}^2\bigr] \,\d\Q(\theta)\,, \
\int_{\Theta} \Tr\bigl(\IFP(\theta)\bigr) \,\d\Q(\theta) < +\infty\,,
\]
where $\Tr$ denotes the trace.
\end{assumption}

\begin{theorem}
\label{th:vT-multi}
The multivariate van Trees inequality~\eqref{eq:vT-multi}
holds with $\IFQ \succ 0$ under Assumption~\ref{ass:def-multidim}
and the fact that $q(\theta) \to 0$
as $\theta$ approaches any boundary point of $\Theta$ with finite norm
along some canonical direction.
\end{theorem}

\subsection{Proof of Theorem~\ref{th:vT-multi}}

Up to resorting to dominated-convergence arguments
(as in Section~\ref{sec:prep-vT1}),
we may restrict our attention to statistics $S$
and to target functions $\psi$ that are uniformly bounded.

\subsubsection{Elements to perform integration by parts.}
\label{sec:IbP-multi}
The key to extend the univariate proof to a multivariate setting
is the following lemma of integration by parts,
which follows from a version of Stokes' theorem
tailored to our needs. Its proof and some comments may be found in
appendix.

\begin{lemma}
\label{lm:IbP}
Let $D \subseteq \R^d$ be an open domain, where $d \geq 2$, and
let $f,\,g : D \to \R$ be two functions that are nice on $D$,
with $g$ also being continuous, such that,
for some $i \in \{1,\ldots,d\}$,
\[
\int_D | f\,g |\,\d\Leb < + \infty \ \ \mbox{and} \ \
\int_D | \partial_i f\,g + f\,\partial_i g|\,\d\Leb < + \infty\,,
\]
and $f(u)\,g(u) \to 0$ as $u$ approaches any boundary point of $D$ with finite norm
along the $i$--th canonical direction. Then
\[
\int_{D \cap \{ g \ne 0 \}} (\partial_i f \,g + f \, \partial_i g) \,\d\Leb = 0\,.
\]
\end{lemma}

Denote by $\psi = (\psi_1,\ldots,\psi_s)$ and $S = (S_1,\ldots,S_s)$
the components of $\psi$ and $S$.
Given the assumptions of Theorem~\ref{th:vT-multi}
and the boundedness of $\psi$, we may directly apply Lemma~\ref{lm:IbP} to $D = \Theta$ and
the pairs $f = \psi_j$ and $g = q$, where $j \in \{1,\ldots,s\}$.

We wish to also do so with $D = \Theta \cap \Supp(q)$ and the $f = \gamma_{S_j}$, where
$\gamma_{S_j} : \theta \in \Theta \mapsto \E_\theta[S_j]$.
The boundary of $\Theta \cap \Supp(q)$ is included in the union of the boundaries
of $\Theta$ and $\Supp(q)$, and $q$ vanishes when it approaches any of them.
Together with the uniform boundedness of $S_j$, the boundary assumption
of Lemma~\ref{lm:IbP} is satisfied on $D = \Theta \cap \Supp(q)$.
It only remains to show that $f = \gamma_{S_j}$ is nice.
To do so, we mimic and adapt arguments used in Section~\ref{sec:loc-int}.
Given that $S_j$ is uniformly bounded, and
given the $\L^2(\mu)$--differentiability assumptions on the model,
we may apply Lemma~\ref{lm:info-eq} along any canonical direction
and get that the $\gamma_{S_j}$ are differentiable in the $i$--th coordinate at all
$\theta \in D$, with partial derivatives given by
\begin{equation}
\label{eq:lm:info-eq--cano}
\partial_i \gamma_{S_j}(\theta) = 2 \int \dxi_{\theta,i} \, \xi_{\theta,i} \, S_j \,\d\mu\,.
\end{equation}
Denoting by $B$ a uniform bound on the $S_j$,
the Cauchy--Schwarz inequality guarantees that
\[
\bigl| \partial_i \gamma_{S_j}(\theta) \bigr|
\leq B \sqrt{\Tr\bigl(\IFP(\theta)\bigr)}
\leq B \Bigl( 1 + \Tr\bigl(\IFP(\theta)\bigr) \Bigr)\,.
\]
Given the final integrability condition in Assumption~\ref{ass:def-multidim}
and the fact that $q$ is nice,
by Fubini's theorem, at almost all $\theta_{-i}$,
the function
\[
\theta_i \in D \mapsto
\Tr\bigl(\IFP(\theta_i,\theta_{-i})\bigr) \, q(\theta_i,\theta_{-i})
\]
is integrable and $\theta_i \in D \mapsto q(\theta_i,\theta_{-i})$
is (absolutely) continuous, thus locally larger than some $\delta > 0$;
recall indeed that $D = \Theta \cap \Supp(q)$ here.
Thus, $\theta_i \in D \mapsto
\Tr\bigl(\IFP(\theta_i,\theta_{-i})\bigr)$ is locally integrable.
Therefore,
at these $\theta_{-i}$, the function $\theta_i \in D(\theta_{-i}) \mapsto \gamma_{S_j}(\theta_i,\theta_{-i})$
is differentiable everywhere, with a derivative that is finite everywhere and locally integrable,
thus (see again \citealp[page 368, {\S}11.83]{Tit39}), it is absolutely continuous.
This exactly corresponds to the fact that $\gamma_{S_j}$ is nice on $D$.

\subsubsection{Brief rest of proof of Theorem~\ref{th:vT-multi}.}
We follow the same methodology as in Section~\ref{sec:proof-vt1},
and introduce
\[
\Delta(x,\theta) =
\nabla{q}(\theta) \frac{\bone_{\{ q(\theta) > 0 \}}}{2\sqrt{q(\theta)}} \, \xi_\theta(x)
+ \sqrt{q(\theta)} \,\, \dxi_\theta(x)\,.
\]
All integrands in the sequel belong to $\L^1(\Leb \otimes \mu)$,
as follows from applications of
the Cauchy--Schwarz inequality. Hence, integrals of sums equal sums of integrals and Fubini's theorem may be applied
to exchange orders of integration.
We use again the short-hand notation $\mu[f]$ for the expectation of a function $f : \cX \to \R^d$ under~$\mu$.

We show below that the
multivariate van Trees inequality~\eqref{eq:vT-multi}
corresponds to
\begin{multline*}
\int_{\cX \times \Theta} \bigl( V(x,\theta) \otimes V(x,\theta) \bigr)\,\d\mu(x)\d\theta \succcurlyeq 0\,, \\
\mbox{where} \ \
V(x,\theta) = \left[ \begin{array}{c} \!\! \bigr( S(x) - \psi(\theta) \bigl) \xi_\theta(x) \sqrt{q(\theta)} \! \vspace{.15cm} \\ 2\Delta(x,\theta) \end{array} \right].
\end{multline*}

We start with the cross-products.
As explained above,
Lemma~\ref{lm:info-eq} may be applied along all canonical directions $i \in \{1,\ldots,p\}$
to yield~\eqref{eq:lm:info-eq--cano} as well as
$\mu\bigl[\dxi_{\theta,i} \, \xi_{\theta,i}\bigr] = 0$
for all $\theta \in \Theta \cap \Supp(q)$. We therefore obtain
the following extension of the four equalities of the beginning
of Section~\ref{eq:IPP-1dim}:
with the short-hand notation $\Theta_q = \Theta \cap \Supp(q)$,
\begin{align*}
\nonumber
& 2 \!\! \bigintsss_{\cX \times \Theta}
\Bigl( \Delta(x,\theta) \otimes \bigl( S(x) - \psi(\theta) \bigr) \Bigr) \xi_\theta(x) \sqrt{q(\theta)} \, \d\mu(x) \d\theta \\
\nonumber
& = \phantom{-} \int_{\Theta_q} \nabla{q}(\theta) \otimes \gamma_S(\theta) \,\d\theta
+ \int_{\Theta_q} \nabla\gamma_S(\theta) \, q(\theta) \,\d\theta \\
& \phantom{=} - \int_{\Theta_q} \nabla{q}(\theta) \otimes \psi(\theta) \,\d\theta
+ (0,\ldots,0)^{\transp}\,,
\end{align*}
where $\gamma_S = (\gamma_{S_j})_{1 \leq j \leq s}$ and $\nabla\gamma_S(\theta)$
is the $p \times s$ matrix whose component $(i,j)$ equals
$\nabla\gamma_S(\theta)_{i,j} = \partial_i \gamma_{S_j}(\theta)$.
The results of Section~\ref{sec:IbP-multi}
hold for all pairs $(i,j)$ and thus guarantee that
\begin{align*}
\int_{\Theta_q} \nabla{q}(\theta) \otimes \gamma_S(\theta) \,\d\theta
= - \int_{\Theta_q} \nabla\gamma_S(\theta) \, q(\theta) \,\d\theta\,, \\
- \int_{\Theta_q} \nabla{q}(\theta) \otimes \psi(\theta) \,\d\theta
= \int_{\Theta_q} \nabla{\psi}(\theta) \, q(\theta) \,\d\theta\,.
\end{align*}

On the other hand, using again that $\mu\bigl[\dxi_{\theta,i} \, \xi_{\theta,i}\bigr] = 0$
for all $\theta \in \Theta_q$, we have that
\[
\int_{\Theta} \nabla{q}(\theta) \otimes
\mu\bigl[ \xi_\theta \dxi_\theta \bigr]\, \bone_{\{ q(\theta) > 0 \}} \, \d\theta
= (0,\ldots,0)^{\transp}\,,
\]
so that
the bottom-right term in the
multivariate van Trees inequality~\eqref{eq:vT-multi}
corresponds to
\begin{align*}
& 4 \int_{\cX \times \Theta} \bigl( \Delta(x,\theta) \otimes \Delta(x,\theta) \bigr) \, \d\mu(x) \d\theta \\
= \ \ & \phantom{+}
\int_{\Theta} \frac{\nabla{q}(\theta) \otimes \nabla{q}(\theta)}{q(\theta)}
\, \bone_{\{ q(\theta) > 0 \}} \, \mu[f_\theta] \, \d\theta \\
& + \int_{\Theta} 4 \, \mu[\dxi_\theta \otimes \dxi_\theta] \,  q(\theta) \,\d\theta \\
= \ \ & \IFQ + \int_{\Theta} \IFP(\theta) \,q(\theta)\,\d\theta\,,
\end{align*}
where $\IFQ \succ 0$ as $q$ cannot be a uniform density due to the boundary conditions.

\section{Direct proof of LAM lower bounds}
\label{sec:LAM-direct}

\citet[Section~3]{GiLe95} provide a derivation of a version of the the H{\'a}jek--Le Cam convolution theorem (\citealp{Haj70})
based on the van Trees inequality.
In the exact same vein, including the same techniques,
we propose a version of the H{\'a}jek--Le Cam local asymptotic minimax [LAM]
theorem (\citealp{Haj72}): see Theorem~\ref{th:LAM} below. We state it
in a H{\'a}jek--Le Cam spirit, avoiding any classic regularity assumption
(contrary to \citealp[Section~3]{GiLe95}).

Its derivation is elementary and bypasses the typical arguments 
of the H{\'a}jek--Le Cam theory of convergence of experiments.
However,
our version requires, on many aspects, stronger assumptions
than the original references, except for the differentiability
of the model, which we only require along canonical directions
(and not in all directions). See the comments after the statement of
Theorem~\ref{th:LAM} for more detail.

\subsubsection*{Setting.}
We still consider an open subset $\Theta \subseteq \R^p$.
For $n \geq 1$, we denote by $\P_\theta^{\otimes n}$ the law of a
$n$--sample of observations based on some $\P_\theta$, and
$\cP^{\otimes n} = (\P_\theta^{\otimes n})_{\theta \in \Theta}$
the associated statistical product model.
When the base statistical model
$\cP$ is differentiable in $\L_2(\mu)$ at some $\theta_0 \in \Theta$ along canonical directions,
then so is $\cP^{\otimes n}$, with a vector of $\L_2(\mu)$-partial derivatives given by
\[
(x_1,\ldots,x_n) \in \cX^n \mapsto \!\left(
\sum_{k=1}^n \dxi_{\theta_0,i}(x_k) \prod_{k' \ne k} \xi_{\theta_0}(x_{k'}) \!\! \right)_{\!\! 1 \leq i \leq p} \!.
\]
In particular, the Fisher information of the product model $\cP^{\otimes n}$ at $\theta_0$ equals
$\IFPn(\theta_0) = n \, \IFP(\theta_0)$.

Consider some sequence of statistics $S_n : \cX^n \to \R^s$
and fix for now some vector $U \in \R^s$.
We assume the following.

\begin{assumption}
\label{ass:LAM}
For a neighborhood $N$ of $\theta_0\in\Theta$,
on the one hand, $\cP$ is differentiable in $\L_2(\mu)$ along
canonical directions at all $\theta \in N$,
and on the other hand, the $\R^s$--valued target function $\psi$ is nice and bounded on $N$,
with $\nabla \psi$ also bounded on $N$.
\end{assumption}

\subsubsection*{Derivation.}
For any distribution $\H$ on $\R^p$,
we denote by $\Q_{\theta_0,r}$
the distribution of $\theta_0 + r H$, where $H$ is a random
variable with distribution~$\H$.
There exist sufficiently regular priors $\H$ on $\R^p$, with support
in the unit ball $\cB$, so that, for all $c > 0$, all assumptions of
Theorem~\ref{th:vT-multi} are satisfied with $\Q = \Q_{\theta_0,c/\sqrt{n}}$, at least
for $n$ large enough (depending on~$\H$ and~$c$),
except maybe the finiteness of the two integrals stated in Assumption~\ref{ass:def-multidim}
(without which the inequality holds also but is pointless).
Also, the Fisher information of $\Q_{\theta_0,c/\sqrt{n}}$
equals $(n/c^2)$ times the Fisher information $\IFQb$ of~$\Q_{\theta_0,1}$.

Therefore, for such priors and for $n$ large enough,
\begin{align*}
& \!\! \bigints_{\cB} \E^{\otimes n}_{\theta_{0}+ c h/\sqrt{n}} \! \left[ \! \Biggl(
U^{\transp} \! \biggl( S_{n} - \psi \Bigl( \theta_{0} + c h/\sqrt{n} \Bigr) \! \biggr) \! \Biggr)^{\!\! 2} \,
\right] \! \d\H(h) \\
& \geq \frac{1}{n} \, U^{\transp} G(\theta_0,c,n)^{\transp} \, I(\theta_0,c,n)^{-1} \, G(\theta_0,c,n) \, U\,,
\end{align*}
where we introduced the $p \times s$ and $p \times p$ matrices
\begin{align*}
G(\theta_0,c,n) & = \int_{\cB} \nabla{\psi}\bigl(\theta_{0}+ c h/\sqrt{n} \bigr) \,\d\H(h)\,, \\
I(\theta_0,c,n) & = \frac{1}{c^2} \, \IFQb + \int_{\cB} \IFP \bigl( \theta_{0}+ c h/\sqrt{n} \bigr) \,\d\H(h)\,.
\end{align*}

Now, any positive quadratic form $\ell : \R^s \to [0,+\infty)$ can be decomposed as follows:
there exists an orthogonal basis $U_1,\ldots,U_s$ of $\R^s$ and nonnegative real numbers $\lambda_1,\ldots,\lambda_s \geq 0$
such that for all $v \in \R^s$,
\[
\ell(v) = \sum_{k=1}^s \lambda_k \bigl( U_k^{\transp} \, v \bigr)^2 =
\sum_{k=1}^s \lambda_k \, U_k^{\transp} \, v\, v^{\transp} \, U_k\,.
\]
This decomposition entails that for all $s \times s$ symmetric positive semi-definite
matrices $\Gamma$, denoting by $\cN\bigl([0],\Gamma\bigr)$
the Gaussian distribution over $\R^s$ centered at $[0] = (0,\ldots,0)^{\transp}$
and with covariance matrix $\Gamma$,
\[
\int_{\R^s} \ell(v) \, \d\cN\bigl([0],\Gamma\bigr)(v) =
\sum_{k=1}^s \lambda_k \, U_k^{\transp} \, \Gamma \, U_k\,.
\]
Linear combinations of the applications above of the van Trees inequality
thus yield
\begin{align*}
& \!\! \bigints_{\cB} \E^{\otimes n}_{\theta_{0}+ c h/\sqrt{n}} \! \left[ \ell \Biggl( \!
\sqrt{n} \biggl( S_{n} - \psi \Bigl( \theta_{0} + c h/\sqrt{n} \Bigr) \! \biggr) \! \Biggr) \!
\right] \! \d\H(h) \\
& \geq \int_{\R^s} \ell(v) \, \d\cN\bigl([0], \Gamma_{\theta_0,c,n} \bigr)(v) \,,
\end{align*}
where $\Gamma_{\theta_0,c,n} = G(\theta_0,c,n)^{\transp} \, I(\theta_0,c,n)^{-1} \, G(\theta_0,c,n)$.
By lower bounding a supremum by an integral,
we obtain the desired LAM lower bound~\eqref{eq:LAM} below as soon as
$\Gamma_{\theta_0,c,n}$ converges in the following sense.
We recall that we do not aim for minimal assumptions in this section, but for elementary arguments.

\begin{assumption}
\label{ass:LAM2}
We have the component-wise convergence
\[
\lim_{c \to \infty} \ \lim_{n \to \infty} \Gamma_{\theta_0,c,n} = \nabla{\psi}(\theta_{0})^{\transp} \IFP(\theta_{0})^{-1} \,
\nabla{\psi}(\theta_{0})\,.
\]
It holds, in particular, as soon as
$\nabla \psi$ and $\IFP$ are continuous at $\theta_0$,
with $\IFP(\theta_0)$ being nonsingular.
\end{assumption}

\begin{theorem}
\label{th:LAM}
Under Assumptions~\ref{ass:LAM} and~\ref{ass:LAM2},
for all positive quadratic forms $\ell : \R^s \to [0,+\infty)$,
for all sequences of statistics $S_n : \cX^n \to \R^s$,
\begin{multline}
\label{eq:LAM}
\liminf_{c \to +\infty} \ \liminf_{n \to +\infty} \ \sup_{\theta : \sqrt{n} \Arrowvert \theta - \theta_0 \Arrowvert \leq c}
\E^{\otimes n}_{\theta} \! \biggl[ \ell \Bigr(
\sqrt{n} \bigl( S_{n} - \psi(\theta) \bigr) \! \Bigr) \!
\biggr]
\\
\geq \int_{\R^s} \ell(v) \, \d\cN\bigl([0], \nabla{\psi}(\theta_{0})^{\transp} \IFP(\theta_{0})^{-1} \,
\nabla{\psi}(\theta_{0}) \bigr)(v) \,.
\end{multline}
\end{theorem}

\subsubsection*{Comments.}
\citet[Theorem~8.11]{vdV98} states the lower bound~\eqref{eq:LAM}
for so-called bowl-shaped loss functions (not just quadratic forms),
under the $\L^2(\mu)$--differentiability of $\cP$ at $\theta_0$ (only,
not on a neighborhood thereof) in all directions (while
Theorem~\ref{th:LAM} considered canonical directions only),
and for $\psi$ differentiable at $\theta_0$ (in sharp contrast with the continuity and boundedness
assumptions on $\psi$ and $\IFP$ in Theorem~\ref{th:LAM}).
That Theorem~\ref{eq:LAM} may only deal with quadratic forms
is unsurprising, given the quadratic nature of the van Trees inequality.
But it came to us as a surprise that the results of Section~\ref{sec:multivar} and thus
Theorem~\ref{eq:LAM} hold for differentiability assumed only along canonical directions.

\begin{remark}
The non-singularity of $\IFP(\theta_0)$ in Assumption~\ref{ass:LAM2} is actually not required
to get a meaningful LAM bound from the van Trees inequality. 
We consider, for instance, the case of $\psi(\theta) = \theta$ and 
only assume that $\IFP$ is continuous at $\theta_0$: the $I(\theta_0,c,n)$ still
converge to $\IFP(\theta)$, which may however be singular.
Now, the proof above reveals that if $U \in \R^p$ is in the kernel of $\IFP(\theta)$,
then the LAM lower bound in~\eqref{eq:LAM} with $\ell(v) = (U^{\transp} v)^2$
equals $+\infty$. 
Conversely, still under the continuity assumption of $\IFP$ at $\theta_0$,
if there exists an estimator having a finite local asymptotic maximum in quadratic risk $\ell(v) = \norm{v}^2$, 
as in the left-hand side of~\eqref{eq:LAM},
then $\IFP(\theta_0)$ is non singular. This can be used to get a simple proof of the non singularity of the efficient Fisher information in semiparametric estimation problems: such an argument has been used in \citet{GaRoVe18} by applying a preliminary version of the 
proof of Theorem~\ref{th:LAM}.
\end{remark}

\begin{appendix}
\section*{Appendix: Proof of Lemma~\ref{lm:IbP}}

We consider the following version of Stokes' theorem,
where we use again the notation of Definition~\ref{def:nice}.
Lemma~\ref{lm:IbP} follows from it by considering
the set $\cS = \{g \ne 0 \}$ and the
product $\varphi = f g$, which is nice as absolute continuity in the classical sense is itself stable by products
(\citealp[page 375, {\S}12.11]{Tit39}).
By continuity of $g$, the set $\cS$ is open and $g$ vanishes at its boundary,
while $f$ (because it is nice) is such that
$f(\,\cdot\,,u_{-i})$ is locally bounded
for almost all $u_{-i} \in D_{-i}$.

\begin{lemma}
\label{lm:IPP}
Let $D \subseteq \R^d$ be an open domain.
Fix a nice function $\varphi : \Theta \to \R$
and $i \in \{1,\ldots,d\}$ such that
\[
\int_D |\varphi| \,\d\Leb < +\infty
\ \ \mbox{and} \ \
\int_D |\partial_i \varphi| \,\d\Leb < +\infty\,,
\]
and such that $\varphi(u)$ tends to 0 as $u$ approaches any boundary point of $D$ with finite norm
along the $i$--th canonical direction.
Consider an open subset $\cS$ such that for almost all $u_{-i} \in D_{-i}$,
one has $\varphi(u_i,u_{-i}) \to 0$ as $u_i$
approaches a boundary point of $\cS$ located in the interior of $D(u_{-i})$.
Then,
\[
\int_{D \cap \cS} \partial_i \varphi\,\d\Leb = 0\,.
\]
\end{lemma}

\begin{proof}
We introduce $G = D \cap \cS$.
By Fubini's theorem, it suffices to show that for almost all $u_{-i} \in G_{-i}$,
\[
\int_{G(u_{-i})} \partial_i \varphi(u_i,u_{-i}) \,\d u_i = 0\,.
\]
Now almost all $u_{-i} \in D_{-i}$ are such that the following holds:
as $\varphi$ is nice on $D$,
$\varphi(\,\cdot\,,u_{-i})$ is absolutely continuous on the open
domains $D(u_{-i})$ and $G(u_{-i})$;
by Fubini's theorem,
\begin{align*}
& \int_{D(u_{-i})} \bigl| \varphi(u_i,u_{-i}) \bigr| \,\d u_i < +\infty \\
\mbox{and} \ \
& \int_{D(u_{-i})} \bigl| \partial_i \varphi(u_i,u_{-i}) \bigr| \,\d u_i < +\infty\,;
\end{align*}
by the $\cS$ boundary assumption, $\varphi(u_i,u_{-i}) \to 0$ as $u_i$
approaches a boundary point of $\cS$ located in the interior of $D(u_{-i})$.
We consider such a point $u_{-i} \in G_{-i}$
and mimic the one-dimensional arguments located in the second part
of Section~\ref{eq:IPP-1dim}. Namely,
we write $G(u_{-i})$ as an (at most) countable disjoint union of open intervals,
\[
G(u_{-i}) = \bigsqcup_{n \geq 1} \, \bigl( a_n(u_{-i}), \, b_n(u_{-i}) \bigr)\,,
\]
where $a_n(u_{-i}) \in \R \cup \{ -\infty \}$ and $b_n(u_{-i}) \in \R \cup \{ +\infty \}$.
By absolute continuity in the classical sense, for all $n \geq 1$,
for all real numbers $a > a_n(u_{-i})$ and $b < b_n(u_{-i})$,
\[
\int_a^b \partial_i{\varphi}(u_i,u_{-i}) \, \d u_i
= \varphi(b,u_{-i}) - \varphi(a,u_{-i})\,.
\]
The boundary of $G$ is included in union of the boundaries of $D$
and $\cS$.
The $D$ and $\cS$ boundary assumptions on $\varphi$ ensure
$\varphi(a,u_{-i}) \to 0$ and $\varphi(b,u_{-i}) \to 0$
as $a \to a_n(u_{-i})$ and $b \to b_n(u_{-i})$, except maybe in the cases where
$a_n(u_{-i}) = -\infty$ or $b_n(u_{-i}) = +\infty$.
In the latter cases, we use that by integrability of $\varphi(\,\cdot\,,u_{-i})$
over $D(u_{-i})$,
the liminf of this function must be null and let $a \to a_n(u_{-i})$ or $b \to b_n(u_{-i})$ in a careful way.
In all cases,
\[
\int_{a_n(u_{-i})}^{b_n(u_{-i})}
\partial_i{\varphi}(u_i,u_{-i}) \, \d u_i = 0
\]
and may sum the obtained equalities over $n \geq 1$, by dominated convergence,
to get the equality claimed at the beginning of this proof.
\end{proof}

\end{appendix}

\begin{acks}[Acknowledgments]
The authors would like to thank David Pollard for
suggesting to study the van Trees inequality under the angle
of a Cram{\'e}r--Rao bound for a location model,
and for following and encouraging this work since 2001,
when he delivered a series of lectures during the statistics semester
at Institut Henri Poincar{\'e}, Paris.
\end{acks}

\begin{funding}
Elisabeth Gassiat was supported by Institut Universitaire de France and by ANR grants
ANR-21-CE23-0035-02 and ANR-23-CE40-0018-02.
\end{funding}

\bibliographystyle{imsart-nameyear}
\bibliography{biblio-vT}

\end{document}